\documentclass[11pt,a4paper]{amsart}

\usepackage{amsthm, amsfonts, amssymb, amsmath, latexsym, enumerate, enumitem, array}
\usepackage[all]{xy}
\usepackage[utf8]{inputenc}
\usepackage{hyperref,cite,mdwlist,mathrsfs}
\SelectTips{cm}{}

\usepackage{amssymb,eucal,graphicx}
\usepackage{enumerate}
\usepackage{colortbl}
\usepackage{xcolor}
\usepackage{mathtools}

\usepackage{hyperref}
\usepackage{epigraph}
 \newtheorem*{akn}{Acknowledgments}
 \newenvironment{akn*}{\begin{akn}\em}{\end{akn}}

\definecolor{darkgreen}{rgb}{.1,.7,.3}

\numberwithin{equation}{section}
\newtheorem{theo}{Theorem}[section]

\newtheorem{prop}[theo]{Proposition}
\newtheorem{dfntn}[theo]{Definition}
\newtheorem{rem}[theo]{Remark}
\newtheorem{claim}[theo]{Claim}

\newtheorem{ass}[theo]{Assumption}
\newtheorem{setup}[theo]{Set-up}
\newenvironment{setup*}{\begin{setup}\em}{\end{setup}}
\newtheorem{case1}[theo]{Case I:}
\newtheorem{case2}[theo]{Case II:}
\newtheorem{ex}[theo]{Example}
\newtheorem{fact}[theo]{FACT}
\newtheorem{facts}[theo]{FACTS}
\newenvironment{rem*}{\begin{rem}\em}{\end{rem}}
\newenvironment{ex*}{\begin{ex}\em}{\end{ex}}
\newenvironment{claim*}{\begin{claim}\em}{\end{claim}}
\newenvironment{facts*}{\begin{facts}\em}{\end{facts}}
\newenvironment{fact*}{\begin{fact}\em}{\end{fact}}
\newenvironment{case1*}{\begin{case1}\em}{\end{case1}}
\newenvironment{case2*}{\begin{case2}\em}{\end{case2}}

\newcommand{\num}{\equiv}

\newcommand{\PP}{\mathbb P}
\newcommand{\Oc}{\mathcal O}
\newcommand{\sA}{\mathscr{A}}
\newcommand{\sB}{\mathscr{B}}

\newcommand{\FF}{\mathbb{F}}

\newcommand{\Ee}{\mathcal{E}_e}

\newcommand{\E}{\mathcal{E}}

\newcommand{\cH}{\mathcal{H}}
\newcommand{\cE}{\mathcal{E}}

\newcommand{\cL}{\mathcal{L}}

\newcommand{\cF}{\mathcal{F}}
\newcommand{\cM}{\mathcal{M}}
\newcommand{\cG}{\mathcal{G}}

\newcommand{\cU}{\mathcal{U}}
\newcommand{\cO}{\mathcal{O}}

\makeatletter
\@namedef{subjclassname@2020}{%
  \textup{2020} Mathematics Subject Classification}
\makeatother

\allowdisplaybreaks[1]

\title{A note on some moduli spaces of Ulrich Bundles}

\author{Maria Lucia Fania}
\address{Maria Lucia Fania\\ Dipartimento di Ingegneria e Scienze dell'Informazione e Matematica\\
Universit\`{a} degli Studi di L'Aquila\\
Via Vetoio Loc. Coppito\\67100 L'Aquila\\Italy}
\email{marialucia.fania@univaq.it}

\author{Flaminio Flamini}
\address{Flaminio Flamini\\Dipartimento di Matematica\\ Universit\`a degli Studi di Roma
Tor Vergata \\ Viale della Ricerca Scientifica, 1 - 00133 Roma\\Italy}
\email{flamini@mat.uniroma2.it}


\thanks{{\bf Acknowledgments}. The first author has been supported by PRIN 2017SSNZAW. The second author has been partially supported by  the MIUR Excellence Department Project MatMod@TOV  MIUR CUP-E83C23000330006, 2023-2027,  awarded to the Department of Mathematics, University of Rome Tor Vergata.  Both authors are members of  INdAM-GNSAGA.  The authors would like to thank M. Aprodu, since this article came about as a result of some of his questions posed during the conference  ``Algebraic Geometry in L'Aquila", July 18-21, 2023, regarding possible irreducibility of the moduli spaces studied in \cite{fa-fl2}. Finally, the authors would like to thank the anonymous Referee for his/her advices, questions and remarks which have also improved the presentation of the present paper.}

\begin{document}
%%%%%%%%%%%%%%%%%%%%%%%%%%%%%%%%%%%%%%%%%%%%%%%%%%%%%%%%%%%%%%%%%%%%%%%%%%%%%%
%%%%%%%%%%%%%%%%%%%% Author(s) and Address %%%%%%%%%%%%%%%%%%%%%%%%%%%%%%%%%%%
%%%%%%%%%%%%%%%%%%%%%%%%%%%%%%%%%%%%%%%%%%%%%%%%%%%%%%%%%%%%%%%%%%%%%%%%%%%%%%

\maketitle

%%%%%%%%%%%%%% ABSTRACT %%%%%%%%%%%%%%%%%%%%%%%%%%%%%%%%%%%%%
\begin{center}
{\it  To Enrique Arrondo, in the occasion of his  60th birthday}
\end{center}

\begin{abstract}  We prove that the modular component $\mathcal M(r)$, constructed  in the Main  Theorem in  \cite{fa-fl2}, of Ulrich vector bundles of  rank $r$ and given Chern classes, on suitable $3$-fold scrolls $X_e$ over Hirzebruch surfaces $\mathbb{F}_{e\ge 0}$,  which arise as tautological embeddings of  projectivization of very-ample vector bundles  on $\mathbb{F}_e$,
is generically smooth, irreducible and  unirational.   A stronger result holds for the suitable associated moduli space $\mathcal M_{\FF_e}(r)$  of  vector bundles of  rank $r$ and given Chern classes on $\mathbb{F}_e$,  Ulrich w.r.t. the very ample polarization $c_1({\cE}_e) = \Oc_{\FF_e}(3, b_e),$ which turns out to be generically smooth, irreducible and  unirational. 
\end{abstract}

%%%%%%%%%%%%%%% INTRODUCTION %%%%%%%%%%%%%%%%%%%%%%%%%%%%%

\section*{Introduction} Let $X$ be a smooth irreducible projective variety of dimension $n \geq 1$, polarized by a  very ample divisor $H$ on $X$. The existence  of  vector bundles $\cU$  on $X$ which are 
\emph{Ulrich with respect to  $\mathcal O_X(H)$ } has interested various authors. 

For some specific classes of varieties such problem has being attacked, see for instance \cite{CM,CMP,ant,f-pl,a-c-mr,fa-fl2}. Whenever such bundles do exist, since they are always {\em semistable} (in the sense of Gieseker-Maruyama, cf. also \S\,\ref{section1} below) and also {\em slope-semistable} (cf. \cite[Def.\,2.7,\,Thm.\,2.9-(a)]{c-h-g-s}), one is interested in knowing if these bundles are also {\em stable}, equivalently {\em slope-stable} (cf. \cite[Def.\,2.7,\,Thm.\,2.9-(c)]{c-h-g-s}). Furthermore, from their semi-stability, such rank-$r$ vector bundles give rise to points in a moduli space, say $M:= M^{ss}(r; c_1, c_2, \ldots, c_k)$, where 
$k := {\rm min} \{r,\,n\}$, parametrizing ($S$-equivalence classes of) semistable sheaves of given rank $r$ and given Chern classes $c_i$ on $X$, $1 \leq i \leq k$ (cf. \cite[p.\, 1250083-9]{c-h-g-s}). Therefore, one is also interested e.g. in understanding: whether $M$ contains at least an irreducible component, say 
$\mathcal M(r)$, which is generically smooth, i.e. reduced, or even smooth; to which sheaf on $X$ corresponds the general point of such a component $\mathcal M(r)$; what can be said about the {\em birational geometry} of $\mathcal M(r)$, namely if it is perhaps rational/unirational; finally, if by chance $M$ turns out to be also irreducible, 
that  is, $M = \mathcal M(r)$.

 In this paper we are interested in some  of the aforementioned properties for the moduli spaces of Ulrich vector bundles on a variety    $X_e$  which  is a $3$-fold scroll over a Hirzebruch surface $\mathbb{F}_e$, with $e \geq 0$. 
More precisely  on $3$-fold scrolls $X_e$ arising as embedding, via very-ample tautological line bundles $\Oc_{\PP(\Ee)}(1)$, of projective bundles $\PP(\E_e)$ over 
$\mathbb{F}_e$, where $\mathcal E_e$ are very-ample rank-$2$ vector bundles on  $\mathbb{F}_e$ with Chern classes 
$c_1 (\mathcal E_e)$ numerically equivalent to $3 C_e + b_ef$ and $c_2(\mathcal E_e) =k_e$, where $C_e$ and $f$ are,    as customary, generators 
of ${\rm Num}(\mathbb{F}_e)$ as in \cite[V, Prop.\,2.3]{H}  and where $b_e$ and $k_e$ are integers satisfying some natural numerical conditions. 
We will set $\xi := \mathcal O_{X_e}(1)$ the hyperplane line bundle of the embedded 
$3$-fold scroll, which we will also call {\em tautological polarization of} $X_e$, as 
$(X_e,\xi) \cong (\PP(\mathcal E_e), \Oc_{\PP(\mathcal E_e)}(1))$.

The existence of Ulrich bundles on such $3$-folds $X_e$ has been considered in \cite{fa-fl2}, where  it was proved that $X_e$  does  not support any Ulrich line bundle  w.r.t. $\xi$, unless $e = 0$. As to Ulrich vector bundles of rank  $r \geq 2$, it was proved  in \cite{fa-fl2}  that   the moduli space $M$, in the above sense, arising from  rank-$r$ vector bundles $\cU_r$ on $X_{e\ge 0}$ which are Ulrich w.r.t. $\xi$ and with first Chern class
\begin{eqnarray*}c_1(\cU_r) =
    \begin{cases}
      r \xi + \varphi^*\Oc_{\FF_e}(3,b_e -3) + \varphi^*\Oc_{\FF_e}\left(\frac{r-3}{2}, \frac{(r-3)}{2}(b_e - e -2)\right), & \mbox{if $r$ is odd}, \\
      r \xi + \varphi^*\Oc_{\FF_e}\left(\frac{r}{2}, \frac{r}{2}(b_e-e-2)\right), & \mbox{if $r$ is even}.
    \end{cases}\end{eqnarray*}
    is not empty and it contains a generically smooth component $\mathcal M(r)$ of dimension 
		\begin{eqnarray*}\dim (\mathcal M(r) ) = \begin{cases} \left(\frac{(r -3)^2}{4}+ 2 \right)(6 b_e - 9e -4) + \frac{9}{2}(r-3) (2b_e-3e), & \mbox{if $r$ is odd}, \\
			 \frac{r^2}{4} (6b_e- 9e-4) +1 , & \mbox{if $r$ is even}.
    \end{cases}
    \end{eqnarray*} The general point $[\cU_{r}] \in \mathcal M(r)$  has been proved to 
correspond to a  slope-stable vector bundle, of slope w.r.t. $\xi$ given by 
$\mu(\cU_r) = 8 b_e - k_e - 12 e - 3$ (see Theorem \ref{Main Theorem} below, for more details). 

As a consequence of such result and a natural one-to-one correspondence among rank-$r$ vector bundles on $X_e$, of the form $\xi \otimes \varphi^*(\mathcal F)$, which are Ulrich w.r.t. $\xi$ on $X_e$, and rank-$r$ vector bundles on $\FF_e$, of the form $\mathcal F (c_1(\mathcal E_e))$, which are Ulrich w.r.t. $c_1(\mathcal E_e) = 3 C_e + b_e f$,  in  \cite{fa-fl2} we  have deduced Ulrichness results for vector bundles on the base surface $\FF_e$ with respect to naturally associated very ample polarization $c_1(\mathcal E_e)$, see Theorem \ref{thm:UlrichFe} for more details. 

By  a result  of Antonelli,  \cite[Theorem 1.2]{ant}, if $\cH_r$ is a rank-$r$ vector bundle on $\FF_e$ which is Ulrich with respect to  a very ample polarization  of the form $\Oc_{\FF_e}(a,b)$  and with $c_1(\cH_r)=\Oc_{\FF_e}(\alpha,\beta)$,  then $\cH_r$  must fit into a short exact sequence of the form
\begin{eqnarray*}
%\label{eq:coker}
\xymatrix{ 0 \to \mathcal O_{\FF_e}(a-1, b - e -1)^{\oplus \gamma}\ar[r]^-{\phi} &   \mathcal O_{\FF_e}(a-1, b - e)^{\oplus \delta} \oplus \mathcal O_{\FF_e}(a, b-1)^{\oplus \tau} \to\cH_r  \to 0,}                                  
 \end{eqnarray*}
where $\gamma, \delta$  and  $\tau$ are suitably defined by  $r, \alpha, \beta, a, b, e$ (cfr. \eqref{eq:coker}).
This fact will  be useful in the present note to give further information about our  modular components $\mathcal M(r)$ as in \cite{fa-fl2}.  Our main results in this paper are the following

\bigskip

\noindent
{ \bf Theorem A} (cf. Theorem \ref{Main Theorem1}, below) {\em
 For any integer $e \geq 0$, let  $\mathbb{F}_e$ be the Hirzebruch surface 
 and let $\Oc_{\FF_e}(\alpha,\beta)$ denote the line bundle 
$\alpha C_e + \beta f$ on $\mathbb{F}_e$, where $C_e$ and $f$ are   generators of ${\rm Num}(\mathbb{F}_e)$ (cf. \cite[V, Prop.\,2.3]{H}). Let $(X_e, \xi)$ be a $3$-fold scroll over $\mathbb{F}_e$ as above, 
where $\varphi: X_e \to \FF_e$ denotes the scroll map. Then the moduli space of rank-$r\geq 2$ vector bundles $\cU_r$ on $X_e$ which are Ulrich w.r.t. $\xi$ and with first Chern class 
\begin{eqnarray*}c_1(\cU_r) =
    \begin{cases}
      r \xi + \varphi^*\Oc_{\FF_e}(3,b_e -3) + \varphi^*\Oc_{\FF_e}\left(\frac{r-3}{2}, \frac{(r-3)}{2}(b_e - e -2)\right), & \mbox{if $r$ is odd}, \\
      r \xi + \varphi^*\Oc_{\FF_e}\left(\frac{r}{2}, \frac{r}{2}(b_e-e-2)\right), & \mbox{if $r$ is even}
    \end{cases}\end{eqnarray*}
    is not empty and it contains a generically smooth component $\mathcal M(r)$, which is of dimension 
		\begin{eqnarray*}\dim (\mathcal M(r) ) = \begin{cases} \left(\frac{(r -3)^2}{4}+ 2 \right)(6 b_e - 9e -4) + \frac{9}{2}(r-3) (2b_e-3e), & \mbox{if $r$ is odd}, \\
			 \frac{r^2}{4} (6b_e- 9e-4) +1 , & \mbox{if $r$ is even},
    \end{cases}
    \end{eqnarray*}  (see  Theorem \ref{Main Theorem}) and which is moreover {\em unirational}.   }
   \bigskip 

For the moduli space of  rank-$r \ge 2$  bundles on  $\FF_e$, the base of the scroll $X_e$,  which are  Ulrich w.r.t. the polarization $c_1(\mathcal E_e) = \Oc_{\FF_e}(3, b_e)$,  a stronger  result holds;  precisely 

\bigskip

\noindent
{ \bf Theorem B} (cf. Theorem \ref{thm:moduli_UlrichFe}, below) {\em
 Let $\mathcal M_{\FF_e}(r)$ be  the moduli space of rank-$r$ vector bundles $\cH_{r}$ on $\FF_e$ which are Ulrich w.r.t. $c_1(\E_e)$ and  with first Chern class
\begin{eqnarray*}c_1(\cH_r) =
    \begin{cases}
     \Oc_{\FF_e}(3 (r+1), (r+1)b_e-3) \otimes \Oc_{\FF_e}\left(\frac{r-3}{2}, \frac{(r-3)}{2}(b_e - e -2)\right), & \mbox{if $r$ is odd}, \\
      \Oc_{\FF_e}(3 r, r b_e) \otimes  \Oc_{\FF_e}\left(\frac{r}{2}, \frac{r}{2}(b_e-e-2)\right), & \mbox{if $r$ is even}. 
    \end{cases}\end{eqnarray*} Then $\mathcal M_{\FF_e}(r) $ is generically smooth, of dimension  
    \begin{eqnarray*}\dim (\mathcal M_{\FF_e}(r) ) = \begin{cases} \left(\frac{(r -3)^2}{4}+ 2 \right)(6 b_e - 9e -4) + \frac{9}{2}(r-3) (2b_e-3e), & \mbox{if $r$ is odd}, \\
			 \frac{r^2}{4} (6b_e- 9e-4) +1 , & \mbox{if $r$ is even}, 
    \end{cases} 
\end{eqnarray*}  (see  Theorem \ref{thm:UlrichFe}) and moreover it is {\em irreducible} and  {\em unirational}.

}

\medskip

The above theorems extend unirationality results in  \cite{ant} and \cite{CM}. 
\medskip

The paper is structured as follows. In Section \ref{section1} we fix notation and terminology. In  Section \ref{section2} we recall some  of the known results that we will use throughout  the paper.  In Section \ref{section3} we state and prove our new main results.

%\begin{akn*} 
%\end{akn*}

\section{Notation and terminology} \label{section1}
%%%%%%%%%%%%%%%%%%%%%%%%%%%%%%%%
%%
%% 	RICHIAMI RISULTATI PRECEDENTI
%%
%%%%%%%%%%%%%%%%%%%%%%%%%%%%%%%%

 In this paper we work over $\mathbb C$. All schemes will be endowed with the Zariski topology. We will  interchangeably use the terms rank-$r$ vector bundle on a smooth, projective variety $X$ and rank-$r$ locally free sheaf. In particular, 
 sometimes, to ease some formulas, with a small abuse of notation we  identify divisor classes with the corresponding line bundles, interchangeably using additive and tensor-product notation. The dual bundle of a rank-$r$ vector bundle $\mathcal F$ on $X$ will be denoted by $\mathcal F^{\vee}$; thus, if $L$ is of rank-$1$, i.e. it is a line bundle, we  interchageably use $L^{\vee}$ or $-L$. If $M$ is a {\em moduli space}, parametrizing objects modulo a given equivalence relation, and if $Y$ is a representative of an equivalence class in $M$,  we will denote by $[Y] \in M$ the point corresponding to $Y$. For non-reminded general terminology, we refer the reader to \cite{H}). 

Because our object  will be  Ulrich bundles, we recall  their definition and basic properties. 

\begin{dfntn}\label{def:Ulrich} Let $X\subset \PP^N$ be a smooth, irreducible, projective variety of dimension $n$  and let $H$  be a hyperplane section of $X$.
A vector bundle $\cU$  on $X$ is said to be  {\em Ulrich} with respect to $\mathcal O_X(H)$  if
\begin{eqnarray*}
H^{i}(X, \cU(-jH))=0 \quad \text{for}  \quad i=0, \cdots, n \quad  \text{and} \quad 1 \leq  j \leq   n.
\end{eqnarray*}
 \end{dfntn}

\begin{dfntn}\label{def:special} Let $X\subset \PP^N$ be a smooth, irreducible, projective variety of dimension $n$  polarized by  $\mathcal O_X(H)$,  where $H$ is a hyperplane section of $X$,   and let $\cU$  be  a rank-$2$ vector bundle  on  $X$ which is  {\em Ulrich} with respect to $\mathcal O_X(H)$. Then $\cU$ is said to be {\em special} if $c_1(\cU) = K_{X} +(n + 1)H.$
\end{dfntn}
 
 For the reader's convenience, we briefly remind facts  concerning (semi)stability and slope-(semi)stability properties of Ulrich bundles as in \cite[Def.\;2.7]{c-h-g-s}. Let $X$ be a smooth, irreducible, projective variety and let $\mathcal F$ be a vector bundle on $X$; recall that $\mathcal F$ is said to be {\em semistable} (in the sense of Gieseker-Maruyama) if for every non-zero coherent subsheaf $\mathcal G \subset \mathcal F$, with $0 < {\rm rk}(\mathcal G) := \mbox{rank of} \; \mathcal G < {\rm rk}(\mathcal F)$, the inequality
$\frac{P_{\mathcal G}}{{\rm rk}(\mathcal G)} \leq \frac{P_{\mathcal F}}{{\rm rk}(\mathcal F)}$ holds true, where 
$P_{\mathcal G}$ and $P_{\mathcal F}$ are the {\em Hilbert polynomials} of the sheaves. Furthermore, $\mathcal F$ is {\em stable} if strict inequality above holds.  Similarly, recall that the {\em slope} of a vector  bundle $\mathcal F$ (w.r.t. a given polarization $\mathcal O_X(H)$ on $X$) is defined to be  $\mu(\mathcal F) := \frac{c_1(\mathcal F) \cdot H^{n-1}}{{\rm rk}(\mathcal F)}$; the bundle $\mathcal F$ is said to be {\em $\mu$-semistable}, or even {\em slope-semistable}, if for every non-zero coherent subsheaf $\mathcal G \subset \mathcal F$ with $0 < {\rm rk}(\mathcal G)   < {\rm rk}(\mathcal F)$, one has $\mu (\mathcal G) \leq \mu(\mathcal F)$. The bundle $\mathcal F$ is {\em $\mu$-stable}, or  {\em slope-stable}, if  strict inequality holds. 

The two definitions of (semi)stability are in general related as follows (cf. e.g. \cite[\S\;2]{c-h-g-s}): 
\begin{eqnarray*}\mbox{slope-stability} \Rightarrow \mbox{stability} \Rightarrow \mbox{semistability} \Rightarrow \mbox{slope-semistability}.\end{eqnarray*}
If $\mathcal U$ is in particular a rank-$r$ vector bundle which is Urlich w.r.t. 
$\mathcal O_X(H)$, then $\mathcal U$ is always semistable, so also slope-semistable 
(cf. \cite[Thm.\;2.9-(a)]{c-h-g-s}); moreover, for $\mathcal U$  the notions of stability and slope-stability coincide  (cf. \cite[Thm.\;2.9-(c)]{c-h-g-s}).

 As for the projective variety which will be the support of Ulrich bundles 
we are interested in,  throughout this work we will denote it by   $X_e$  and it will  be a $3$-dimensional scroll over the Hirzebruch surface 
$\FF_e := \PP(\Oc_{\PP^1} \oplus\Oc_{\PP^1}(-e))$, with $e \geq  0$ an integer. 

More precisely, let $\pi_e : \FF_e \to \PP^1$ be the natural projection onto the base. Then,  as in \cite[V, Prop.\,2.3]{H}, 
${\rm Num}(\FF_e) = \mathbb{Z}[C_e] \oplus \mathbb{Z}[f],$ where:

\smallskip

\noindent
$\bullet$ $f := \pi_e^*(p)$, for any $p \in \PP^1$, whereas

\smallskip

\noindent
$\bullet$ $C_e$ denotes either the unique section corresponding to the morphism of vector bundles on $\PP^1$ \linebreak
$\Oc_{\PP^1} \oplus\Oc_{\PP^1}(-e) \to\!\!\!\to \Oc_{\PP^1}(-e)$, when $e>0$, or the fiber of the other ruling different from that 
induced by $f$, when otherwise $e=0$. 

\smallskip

\noindent
In particular$$C_e^2 = - e, \; f^2 = 0, \; C_ef = 1.$$

Let $\Ee$ be a rank-$2$ vector bundle over $\FF_e$ and let $c_i(\mathcal{E}_e)$ be its  $i^{th}$-Chern class. Then $c_1( \mathcal{E}_e) \num a C_e + b f$, for some $ a, b \in \mathbb Z$, and $c_2(\cE_e) \in \mathbb Z$. For the line bundle $\cL \num \alpha C_e + \beta f$ we will also use the notation $\Oc_{\FF_e}(\alpha,\beta)$.

From now on, we will consider the following:

 \begin{ass}\label{ass:AB}  Let $e \geq  0$, $b_e$, $k_e$ be integers such that
\begin{equation}\label{(iii)}
b_e-e< k_e< 2b_e-4e,
\end{equation} and let ${\mathcal E}_e$ be a rank-$2$ vector bundle over $\FF_e$, with
\begin{eqnarray*}
c_1({\cE}_e) \num 3 C_e + b_e f \;\; {\rm and} \;\; c_2({\cE}_e) = k_e,
\end{eqnarray*} which fits in the exact sequence 
\begin{equation}\label{eq:al-be}
0 \to A_e \to {\cE}_e \to B_e \to 0,
\end{equation} where $A_e$ and $B_e$ are line bundles on $\FF_e$ such that
\begin{equation}\label{eq:al-be3}
A_e \num 2 C_e + (2b_e-k_e-2e) f \;\; {\rm and} \;\; B_e \num C_e + (k_e - b_e + 2e) f
\end{equation}
\end{ass} \color{black} From \eqref{eq:al-be}, in particular, one has $c_1({\cE}_e) = A_e + B_e \;\; {\rm and} \;\; c_2({\cE}_e) = A_eB_e$.
\bigskip

By  results in \cite{fa-fl2},   ${\mathcal E}_e$  as above, turns out to be very ample on $\FF_e$.  Thus  we take $X_e$   to be the $3$-fold scroll arising as embedding, via very-ample tautological line bundle $\Oc_{\PP(\Ee)}(1)$, of the projective bundle $\PP(\E_e)$.

\bigskip

\section{Preliminaries}\label{section2}
%%%%%%%%%%%%%%%%
%
%FANIA - LELLI CHIESA - PONS LLOPIS
%
%%%%%%%%%%%%%%%%
In this section, for the reader convenience,  we state some of  the  known results that we will be using in the sequel.

The following Theorem \ref{pullback}, (cf.\!\!\cite[Theorem 2.4]{f-lc-pl}) states under which conditions an Ulrich bundle on the base of the scroll gives rise to a bundle on the scroll itself which is Ulrich w.r.t.  the {\em tautological polarization} 
$\xi$. 

\begin{theo}\label{pullback}{\rm(\!\!\cite[Theorem 2.4]{f-lc-pl})}  Let $(S,H)$ be a polarized surface, with $H$ a very ample line bundle, and let $\mathcal E$ be a rank-$2$ vector bundle on $S$ such that $\mathcal E$ is (very) ample and spanned. 
Let $\mathcal F$ be a rank-$r \geq 1$ vector bundle on $S$. Let $(X,\xi) \cong (\PP(\mathcal E), \Oc_{\PP(\mathcal E)}(1))$ be a $3$-fold scroll over $S$, where 
 $\xi$ is the tautological polarization,  and let $X \xrightarrow{\varphi} S$ denote the scroll map. Then  the vector bundle $\mathcal U:= \xi \otimes \varphi^*(\mathcal F)$ is Ulrich with respect to $\xi$ if and only if the bundle $\mathcal F$ is such that
\begin{equation}\label{needed to pullback}
\begin{array}{ccc}
  H^i(S,\mathcal F)=0 & \text{and} & H^i(S,\mathcal F(-c_1(\mathcal E)))=0, \; 0 \leq i \leq 2.
\end{array}
\end{equation} In particular, if $c_1(\mathcal E)$ is very ample on $S$, then the rank-$r$ vector bundle on $X$, $ \mathcal U=\xi \otimes \varphi^*(\mathcal F)$, is Ulrich with respect to $\xi$  if and only if the rank-$r$ vector bundle on $S$,  $\mathcal F (c_1(\mathcal E))$, is Ulrich with respect to $c_1(\mathcal E)$. 
\end{theo}

\bigskip
Viceversa,  starting with a rank-r vector bundle on the 3-fold scroll $(X, \xi)$ which is Ulrich w.r.t. $\xi$,  satisfying suitable properties, we recall how to obtain an Ulrich vector bundle of the same rank on the base $S$ of the scroll.
\medskip

Let $\varphi:X \rightarrow S$ be a $3$-fold scroll over a surface $S$.  Let us recall, see \cite[Theorem 11.1.2.]{BESO}, that a general hyperplane section $\widetilde{S}$ of $X$ has the structure of a blow-up of the base surface $S$ at $c_2(\cE)$ points and one can consider the following diagram:

\begin{equation}
    \label{composition}
    \xymatrix@-2ex{
    \widetilde{S} \, \, \ar@{^{(}->}[r]^i \ar[rd]_{\varphi'} & X \ar[d]^{\varphi} \\
                                      & S,
    }
\end{equation}
where $i$ is the inclusion and $\varphi'$ is the blow-up map, where we denote by $E_i$ the exceptional divisors of the latter map. 
 More precisely, if $\widetilde{S} \in |\xi|$ is a general hyperplane section of $X$, then it corresponds to the vanishing locus 
of a general global section $\widetilde{\sigma} \in H^0(X, \xi)$; since one has 
$H^0(X, \xi) \cong H^0( \PP(\mathcal E), \Oc_{\PP(\mathcal E)}(1)) \cong H^0(S, \mathcal E)$, then $\widetilde{\sigma}$  bijectively corresponds to a global
section $\sigma$ of $\mathcal E$ whose vanishing locus $Z:= V (\sigma)$ is a zero-dimensional subscheme on $S$ which is an element of $c_2(\mathcal E)$. From \cite[Theorem 11.1.2.]{BESO}, $\widetilde{S}$ turns out to be isomorphic to the blow-up of $\varphi': \widetilde{S} \to S$ at such points 
$Z$ and, for any $z \in Z$, the $\varphi$-fiber $\varphi^{-1}(z) := F_z$ of $X$ is contained in $\widetilde{S}$ as the $\varphi'$-exceptional divisor $E_z$ over the point $z$ of such a blow-up $\varphi'$. 

With this set-up, in \cite[Thm.\,6.1,\,Prop.\,6.2]{f-lc-pl}, the authors gave conditions to get bijective correspondences among rank-$r$ bundles on $X$ which are Ulrich w.r.t. the tautological polarization $\xi$ and rank-$r$ bundles on the base surface $S$ which are Ulrich w.r.t. the naturally related polarization as in Theorem \ref{pullback}. 

\begin{theo}\label{pushforwardulrich}{\rm(\!\!\cite[Theorem 6.1]{f-lc-pl})}
  Let $\varphi: X \rightarrow S$ be a $3$-fold scroll over a surface $S$ and let $\cG$ be a rank-$r$ vector bundle on $X$ which is Ulrich  with respect to the tautological polarization $\xi$, i.e. $(X,\xi) \cong (\PP(\mathcal E), \Oc_{\PP(\mathcal E)}(1))$.   Let us suppose that $c_1(\mathcal E)$ is very ample on $S$. Assume that on the  general fiber $F=\varphi^{-1}(s)$, $s\in S$, the vector bundle $\cG$ splits as follows: $\cG_{|F}\cong\cO_{\PP^1}(1)^{\oplus r}$. Then $\varphi_*(\cG\otimes i_*(\cO_{\tilde{S}}(\sum_{i=1}^k E_i))$, with  $k = |c_2(\mathcal E)|$, is a rank-$r$ vector bundle on $S$ which is Ulrich w.r.t. $c_1(\cE)$.
\end{theo}

In the following remark we comment on the hypotheses of Theorem \ref{pushforwardulrich}, in order to better explain the aforementioned Ulrich-bundle bijective correspondence arising from Theorems \ref{pullback} and \ref{pushforwardulrich} (cf. Proposition \ref{bijection} below).

\begin{rem*}\label{Rmk1}  We like to point out that the assumption on the splitting-type of the vector bundle $\cG$ on the general fiber $F$ of $\varphi$ as $\cG_{|F}\cong\cO_{\PP^1}(1)^{\oplus r}$ as in Theorem \ref{pushforwardulrich} implies that such a splitting-type  holds true for all $\varphi$-fibers $\varphi^{-1}(u) := F_u$, for $u$ varying in a suitable open dense subset  $U \subseteq S$. Thus, from the previous description on the birational structure of a general hyperplane section $\widetilde{S} = V(\widetilde{\sigma})$ of $X$ as in \eqref{composition}, the main points to let the Ulrich-bundle bijective correspondence arise are first of all that the zero-dimensional scheme $Z = V (\sigma)$, corresponding to $\widetilde{S} \in |\xi|$ general, is entirely contained in the open set $U \subseteq S$ (so that, for any $z \in Z$,  the restriction of $\mathcal G$ to $F_z := \varphi^{-1}(z)$ is $\cG_{|F_z}\cong\cO_{\PP^1}(1)^{\oplus r}$ namely, from \eqref{composition}, $\cG_{|E_i} \cong \cO_{\PP^1}(1)^{\oplus r}$, for any $1 \leq i \leq |c_2(\mathcal E)|$, where  $\sum_i\, E_i$ denotes the total exceptional divisor of the blow-up $\varphi'$ of $S$ along $Z$) and then the use of  \cite[Thm.\,4.2]{cas-kim}. 
\end{rem*}  
Arguments described in Remark \ref{Rmk1} are the principles  used in \cite{f-lc-pl} to get the following Proposition.

\begin{prop}\label{bijection}{\rm(\!\!\cite[Prop.\,6.2]{f-lc-pl})} Let $\varphi: X \rightarrow S$ be a $3$-fold scroll over a surface $S$, where  $(X,\xi) \cong (\PP(\mathcal E), \Oc_{\PP(\mathcal E)}(1))$ for some very ample rank-$2$ vector bundle $\mathcal E$ on $S$. Assume that $c_1(\mathcal E)$ is very ample on $S$. Then there exists a bijection :

$$
\left\{ \begin{array}{c}
          \text{Bundles $\cF$ of rank $r$ on $S$ } \\
          \text{which are Ulrich w.r.t. $c_1(\cE)$}
        \end{array}\right\}_{\biggm/ \cong_{iso}}\Leftrightarrow\left\{\begin{array}{c}
          \text{Bundles $\cG$ of rank $r$ on $X$} \\
          \text{which are Ulrich w.r.t. $\xi$ and such that} \\
           \text{ $\cG_{|\varphi^{-1}(s)}\cong\cO_{\PP^1}(1)^{\oplus r}$, for  general  $s\in S$}
        \end{array}\right\}_{\biggm/ \cong_{iso}}
$$
the bijection given by the maps
\begin{equation*}
  \phi:\cF\quad\mapsto\quad \cG:= \xi \otimes \varphi^*(\cF(-c_1(\cE)));
\end{equation*}
\noindent and
  \begin{equation*}
  \psi:\cG\quad\mapsto\quad \cF:=\varphi_*(\cG\otimes i_*(\cO_{\tilde{S}}(\sum^k_{i=1} E_i)).
\end{equation*}
\end{prop}

\bigskip

Because we are interested on moduli spaces of Ulrich bundles on $3$-folds scrolls $X_e$ over $\FF_e$, as well as on moduli spaces of Ulrich bundles on  $\FF_e$,  we recall what was already proved in  \cite{fa-fl2}.

\bigskip
%%%%%%%%%%%%%%%%
%
%  FANIA - FLAMINI
%
%%%%%%%%%%%%%%%%
\noindent
\begin{theo}{\rm(\!\!\cite[{\bf Main Theorem}]{fa-fl2})} \label{Main Theorem}
 For any integer $e \geq 0$, consider the Hirzebruch surface 
$\mathbb{F}_e$ and let $\Oc_{\FF_e}(\alpha,\beta)$ denote the line bundle 
$\alpha C_e + \beta f$ on $\mathbb{F}_e$, where $C_e$ and $f$ are   generators  of ${\rm Num}(\mathbb{F}_e)$.

Let $(X_e, \xi)$ be a $3$-fold scroll over $\mathbb{F}_e$ as in Assumption \ref{ass:AB}, 
where $\varphi: X_e \to \FF_e$ denotes the scroll map. Then: 

\smallskip

\noindent 
(a) $X_e$  does  not support any Ulrich line bundle  w.r.t. $\xi$ unless $e = 0$. In this latter case, the unique Ulrich line bundles on 
$X_0$ are the following: 
\begin{itemize}
\item[(i)]  $L_1 :=\xi+\varphi^*\Oc_{\FF_0}(2,-1) $ and $ L_2 :=\xi+\varphi^*\Oc_{\FF_0}(-1,b_0-1)$;  
\item[(ii)]  for any integer $t\geq 1$, $M_1 :=2\xi+\varphi^*\Oc_{\FF_0}(-1,-t-1)$ and $M_2:=\varphi^*\Oc_{\FF_0}(2,3t-1)$, which only occur for $b_0=2t, k_0=3t$.
\end{itemize}

\smallskip 

\noindent
(b) Set $e=0$ and let $r \geq 2$ be any integer. Then the moduli space of rank-$r$ vector bundles $\cU_r$ on $X_0$ which are Ulrich w.r.t. $\xi$ and with first Chern class
\begin{eqnarray*}c_1(\cU_r) =
    \begin{cases}
      r \xi + \varphi^*\Oc_{\FF_0}(3, b_0-3) + \varphi^*\Oc_{\FF_0}\left(\frac{r-3}{2}, \frac{(r-3)}{2}(b_0-2)\right),  & \mbox{if $r$ is odd}, \\
      r \xi + \varphi^*\Oc_{\FF_0}(\frac{r}{2},\frac{r}{2}(b_0-2)), \color{black} & \mbox{if $r$ is even}.
    \end{cases}\end{eqnarray*}
    is not empty and it contains a generically smooth component $\mathcal M(r)$ of dimension 
		\begin{eqnarray*}\dim (\mathcal M(r) ) = \begin{cases} \frac{(r^2 -1)}{4}(6 b_0 -4), & \mbox{if $r$ is odd}, \\
			 \frac{r^2}{4} (6b_0-4) +1 , & \mbox{if $r$ is even}.
    \end{cases}
    \end{eqnarray*}
    The general point $[\cU_r] \in \mathcal M(r)$  
corresponds to a  slope-stable vector bundle, of slope w.r.t. $\xi$ given by 
$\mu(\cU_r) = 8b_0 - k_0 -3$. If moreover $r=2$, then $\cU_2$ is also {\em special} (cf. Def. \ref{def:special} above).

\smallskip

\noindent
(c) When $e >0$, let $r \geq 2$ be any integer. Then the moduli space of rank-$r$ vector bundles $\cU_{r}$ on $X_e$ which are Ulrich w.r.t. $\xi$ and with first Chern class
\begin{eqnarray*}c_1(\cU_r) =
    \begin{cases}
      r \xi + \varphi^*\Oc_{\FF_e}(3,b_e -3) + \varphi^*\Oc_{\FF_e}\left(\frac{r-3}{2}, \frac{(r-3)}{2}(b_e - e -2)\right), & \mbox{if $r$ is odd}, \\
      r \xi + \varphi^*\Oc_{\FF_e}\left(\frac{r}{2}, \frac{r}{2}(b_e-e-2)\right), & \mbox{if $r$ is even}.
    \end{cases}\end{eqnarray*}
    is not empty and it contains a generically smooth component $\mathcal M(r)$ of dimension 
		\begin{eqnarray*}\dim (\mathcal M(r) ) = \begin{cases} \left(\frac{(r -3)^2}{4}+ 2 \right)(6 b_e - 9e -4) + \frac{9}{2}(r-3) (2b_e-3e), & \mbox{if $r$ is odd}, \\
			 \frac{r^2}{4} (6b_e- 9e-4) +1 , & \mbox{if $r$ is even}.
    \end{cases}
    \end{eqnarray*} The general point $[\cU_{r}] \in \mathcal M(r)$  
corresponds to a  slope-stable vector bundle, of slope w.r.t. $\xi$ given by 
$\mu(\cU_r) = 8 b_e - k_e - 12 e - 3$.  If moreover $r=2$, then $\cU_2$ is also special. 
\end{theo}

\bigskip

 We want to stress that in \cite[Proof of Thm.\,5.1]{fa-fl2} it has been proved that bundles $L_1, L_2$ and $\mathcal U_r$, for any $r \geq 2$, as in Theorem \ref{Main Theorem} split on any $\varphi$-fiber of $X_e$ as requested in Theorem \ref{pushforwardulrich} and in Proposition \ref{bijection}, namely for any $\varphi$-fiber $F$, one has 
$(L_1)_{|F} = (L_2)_{|F} \cong \mathcal O_{\mathbb P^1}(1)$ whereas $(\mathcal U_r)_{|F} \cong \mathcal O_{\mathbb P^1}(1)^{\oplus r}$ (this is due to the iterative contructions in \cite{fa-fl2} of such bundles as deformations of iterative extensions).  As a direct consequence of Theorem \ref{Main Theorem},  Theorem \ref{pullback} and the one--to--one correspondence in Proposition \ref{bijection},  in \cite{fa-fl2} we could prove  the following result concerning moduli spaces of rank-$r$ vector bundles on Hirzebruch surfaces $\FF_e$, for any $r \geq 1$ and any $e \geq  0$, which are Ulrich w.r.t. the very ample line bundle $c_1(\mathcal E_e) = 3 C_e + b_ef$, with $b_e \geq  3e+2$ as it follows from  Assumption \ref{ass:AB} (the case $r=1,2,3$ 
already known by \cite{a-c-mr,cas,ant}). 

\begin{theo}\label{thm:UlrichFe} {\rm(\!\!\cite[Theorem 5.1]{fa-fl2})} For any integer $e \geq  0$, consider the Hirzebruch surface 
$\mathbb{F}_e$ and let $\Oc_{\FF_e}(\alpha,\beta)$ denote the line bundle 
$\alpha C_e + \beta f$ on $\mathbb{F}_e$, where $C_e$ and $f$ are  generators  of ${\rm Num}(\mathbb{F}_e)$.

Consider the very ample polarization $c_1(\E_e)= \Oc_{\FF_e}(3,b_e)$ on $\mathbb{F}_e$, where $b_e \geq  3e+2$. Then: 

\smallskip

\noindent 
(a) $\FF_e$  does  not support any Ulrich line bundle  w.r.t. $c_1(\E_e)$ unless $e = 0$. In this latter case, the unique line bundles on 
$\FF_0$ which are Ulrich w.r.t. $c_1(\E_e)$ are 
$$\mathcal L_1 :=\Oc_{\FF_0}(5,b_0-1) \; {\rm and} \; \mathcal L_2 := \Oc_{\FF_0}(2,2b_0-1).$$

\smallskip 

\noindent
(b) Set $e=0$ and let $r \geq 2$ be any integer. Then the moduli space $\mathcal M_{\FF_0}(r)$ of rank-$r$ vector bundles $\cH_r$ on $\FF_0$ which are Ulrich w.r.t. $c_1(\E_0)$ and with first Chern class
\begin{eqnarray*}c_1(\cH_r) =
    \begin{cases}
      \Oc_{\FF_0}(3(r+1), (r+1) b_0 -3) \otimes  \Oc_{\FF_0}\left(\frac{r-3}{2}, \frac{(r-3)}{2}(b_0-2)\right), & \mbox{if $r$ is odd}, \\
      \Oc_{\FF_0}(3r, r b_0) \otimes \Oc_{\FF_0}\left(\frac{r}{2}, \frac{r}{2}(b_0-2)\right), & \mbox{if $r$ is even}.
    \end{cases}\end{eqnarray*}
    is not empty and it contains a generically smooth component  of dimension 
		\begin{eqnarray*} 
		\begin{cases} \frac{(r^2 -1)}{4}(6 b_0 -4), & \mbox{if $r$ is odd}, \\
			 \frac{r^2}{4} (6b_0-4) +1 , & \mbox{if $r$ is even}.
    \end{cases}
    \end{eqnarray*}
    The general point $[\cH_r] $ of such a component  corresponds to a  slope-stable vector bundle.

\smallskip

\noindent
(c) When $e >0$, let $r \geq  2$ be any integer. Then the moduli space  $\mathcal M_{\FF_e}(r)$ of rank-$r$ vector bundles $\cH_{r}$ on $\FF_e$ which are Ulrich w.r.t. $c_1(\E_e)$ and with first Chern class
\begin{eqnarray*}c_1(\cH_r) =
    \begin{cases}
     \Oc_{\FF_e}(3 (r+1), (r+1)b_e-3) \otimes \Oc_{\FF_e}\left(\frac{r-3}{2}, \frac{(r-3)}{2}(b_e - e -2)\right), & \mbox{if $r$ is odd}, \\
      \Oc_{\FF_e}(3 r, r b_e) \otimes  \Oc_{\FF_e}\left(\frac{r}{2}, \frac{r}{2}(b_e-e-2)\right), & \mbox{if $r$ is even}.
    \end{cases}
    \end{eqnarray*}
    is not empty and it contains a generically smooth component  of dimension 
		\begin{eqnarray*}
		 \begin{cases} \left(\frac{(r -3)^2}{4}+ 2 \right)(6 b_e - 9e -4) + \frac{9}{2}(r-3) (2b_e-3e), & \mbox{if $r$ is odd}, \\
			 \frac{r^2}{4} (6b_e- 9e-4) +1 , & \mbox{if $r$ is even}.
    \end{cases}
    \end{eqnarray*} 
    The general point $[\cH_{r}]$ of such a component corresponds to a  slope-stable vector bundle. 
\end{theo}

%%%%%%%%%%%%%%%%%%%%%%
%
% THEOREM 3.1 and THEOREM 3.2
%
%%%%%%%%%%%%%%%%%%%%%%
\section{Moduli  spaces}\label{section3}

Our aim in this section is to prove that the moduli space  $\mathcal M_{\FF_e}(r)$  of Ulrich bundles on $\FF_e$, $e \geq 0$, as  in  Theorem \ref{thm:UlrichFe}  is  irreducible, generically smooth and unirational, whereas that the generically smooth modular  component  $\mathcal M(r)$  of Ulrich bundles on $X_e$, $e \geq 0$, as in Theorem \ref{Main Theorem} is unirational.

%%%%%%%%%%%
%
% THEOREM 3.1 
%
%%%%%%%%%%%
\begin{theo}\label{thm:moduli_UlrichFe} Let $\mathcal M_{\FF_e}(r)$ be  the moduli space of rank-$r$ vector bundles $\cH_{r}$ on $\FF_e$ which are Ulrich w.r.t. $c_1(\E_e)= \Oc_{\FF_e}(3, b_e)$ and  with first Chern class
\begin{eqnarray*}c_1(\cH_r) =
    \begin{cases}
     \Oc_{\FF_e}(3 (r+1), (r+1)b_e-3) \otimes \Oc_{\FF_e}\left(\frac{r-3}{2}, \frac{(r-3)}{2}(b_e - e -2)\right), & \mbox{if $r$ is odd}, \\
      \Oc_{\FF_e}(3 r, r b_e) \otimes  \Oc_{\FF_e}\left(\frac{r}{2}, \frac{r}{2}(b_e-e-2)\right), & \mbox{if $r$ is even},
    \end{cases}\end{eqnarray*}
      (see  Theorem \ref{thm:UlrichFe}).  Then $\mathcal M_{\FF_e}(r) $ is generically  smooth, irreducible,   unirational and of dimension 
    \begin{eqnarray*}\dim (\mathcal M_{\FF_e}(r) ) = \begin{cases} \left(\frac{(r -3)^2}{4}+ 2 \right)(6 b_e - 9e -4) + \frac{9}{2}(r-3) (2b_e-3e), & \mbox{if $r$ is odd}, \\
			 \frac{r^2}{4} (6b_e- 9e-4) +1 , & \mbox{if $r$ is even}.
    \end{cases}
    \end{eqnarray*}
\end{theo}

\begin{proof}
%{\color{red}{$\cH_r $ come coker}}

From Theorem \ref{thm:UlrichFe} we know  that the moduli  space   $\mathcal M_{\FF_e}(r)$ is not empty.  

Let $\mathcal M\subseteq  \mathcal M_{\FF_e}(r)$ be any   irreducible component  and let   $[\cH_r ]\in  \mathcal M$ be  its general point. So  $\cH_r$  is  of rank $r$ and as in  the statement of Theorem \ref{thm:moduli_UlrichFe}. 

For simplicity let $ c_1(\cH_r) =   \Oc_{\FF_e}(\alpha, \beta)$.
    By \cite[Theorem 1.1]{ant} $\cH_r $  necessarily  fits into the following short exact sequence  
\begin{eqnarray}
\label{eq:coker}
\xymatrix{ 0 \to \mathcal O_{\FF_e}(2, b_e - e -1)^{\oplus \gamma}\ar[r]^-{\phi} &   \mathcal O_{\FF_e}(2, b_e - e)^{\oplus \delta} \oplus \mathcal O_{\FF_e}(3, b_e-1)^{\oplus \tau} \to\cH_r  \to 0.}                                  
 \end{eqnarray}
 where  $\gamma=\alpha+\beta-r(2+b_e)-e(\alpha-3r)$,\,  $\delta=\beta-r(b_e-1)-e(\alpha-3r)$,    $\tau=\alpha-2r$ which, after plugging in the value of $\alpha$ and $\beta$, become
\bigskip

$\gamma=\frac{(b_e-2e+1)r-b_e+3}{2}$,  $\delta= \frac{(r-1)b_e}{2}-er$, $\tau=\frac{3(r+1)}{2}$, \, if $r$ is odd, \, and 
 
 $\gamma= \frac{(b_e-2e+1)r}{2}$,  $\delta= \frac{(b_e-2e)r}{2}$, $\tau=\frac{3r}{2}$, if $r$ is even.
 \medskip
 
   Thus   $\cH_r $ is expressed as the cokernel of  an injective map $\phi \in \rm{Hom}_{\FF_e} ( \sA, \sB)$, where
  \linebreak $\sA := \mathcal O_{\FF_e}(2, b_e - e -1)^{\oplus \gamma}$ and $\sB :=  \mathcal O_{\FF_e}(2, b_e - e)^{\oplus \delta} \oplus \mathcal O_{\FF_e}(3, b_e-1)^{\oplus \tau}$, with $\gamma, \delta, \tau$  as above. 
   
   On  the other hand, by  \cite[Theorem 1.3]{ant}, if we  take a general map  $\phi_{gen} \in \rm{Hom}_{\FF_e} ( \sA, \sB)$ then ${\rm coker}(\phi_{gen})$ is   a rank-$r$ vector bundle on $\FF_e$, in particular locally free, which is 
   Ulrich w.r.t. $c_1(\mathcal E_e)$, and with Chern classes 
    $c_1({\rm coker}(\phi_{gen}))$ and $c_2({\rm coker}(\phi_{gen}))$ as those of $\cH_r $. Since $ \sA, \sB$ are uniquely  determined  by $r$, $e$, $(3,b_e)$ and $ c_1(\cH_r)$  and since $\rm{Hom}_{\FF_e} ( \sA, \sB)$ is irreducible, it follows  that   $\mathcal M=\mathcal M_{\FF_e}(r)$, i.e. $\mathcal M_{\FF_e}(r)$ is therefore irreducible and moreover it is unirational, being dominated by  $\rm{Hom}_{\FF_e} ( \sA, \sB)$. 
   
 The generic smoothness of  $\mathcal M_{\FF_e}(r)$ and the formula for its dimension follow as they have already been proved in  Theorem \ref{thm:UlrichFe}-(b), (c). 
   \end{proof}
   %%%%%%%%%%%
%
% THEOREM  3.2
%
%%%%%%%%%%%
   \begin{theo}\label{Main Theorem1}
 For any integer $e \geq 0$, let  $\mathbb{F}_e$ be the Hirzebruch surface 
 and let $\Oc_{\FF_e}(\alpha,\beta)$ denote the line bundle 
$\alpha C_e + \beta f$ on $\mathbb{F}_e$, where $C_e$ and $f$ are  generators of ${\rm Num}(\mathbb{F}_e)$.

Let $(X_e, \xi)$ be a $3$-fold scroll over $\mathbb{F}_e$ as in Assumption \ref{ass:AB},  where $\varphi: X_e \to \FF_e$ denotes the scroll map. Then the moduli space of rank-$r\geq 2$ vector bundles $\cU_r$ on $X_e$ which are Ulrich w.r.t. $\xi$ and with first Chern class as in Theorem \ref{Main Theorem}
    is not empty and it contains a generically smooth component $\mathcal M(r)$ which is unirational  and  of dimension
    \begin{eqnarray*}\dim (\mathcal M(r) ) = \begin{cases} \left(\frac{(r -3)^2}{4}+ 2 \right)(6 b_e - 9e -4) + \frac{9}{2}(r-3) (2b_e-3e), & \mbox{if $r$ is odd}, \\
			 \frac{r^2}{4} (6b_e- 9e-4) +1 , & \mbox{if $r$ is even}.
    \end{cases}
    \end{eqnarray*}  \end{theo}
   
  \begin{proof} 
As we  have seen in  the proof of Theorem \ref{thm:moduli_UlrichFe}, a general $[\cH_r]  \in \mathcal M_{\FF_e}(r)$, turns out to  be $\cH_r ={\rm coker}(\phi)$, with  $\phi$ a general vector bundle morphism as in  \eqref{eq:coker}.

 Now take $\sA = \mathcal O_{\FF_e}(2, b_e - e -1)^{\oplus \gamma}$, $\sB =  \mathcal O_{\FF_e}(2, b_e - e)^{\oplus \delta} \oplus \mathcal O_{\FF_e}(3, b_e-1)^{\oplus \tau}$, $\gamma, \delta$ and $\tau$ as in the proof of Theorem \ref{thm:moduli_UlrichFe}; then for $\phi \in \rm{Hom}_{\FF_e} ( \sA, \sB)$ general, one has therefore 
$$ 0 \to \sA \stackrel{\phi}{\longrightarrow}   \sB \to \cH_r  \to 0.$$
\color{black} We first  tensor this exact sequence by $-c_1(\E_e)$,  then we pull it back via $\varphi^*$, where $\varphi: X_e \to \FF_e$ is the scroll map,  and the sequence remains exact on the left since $\cH_r(- c_1(\mathcal E_e))$ is locally free; \color{black} subsequently we tensor the  resulting short exact sequence with  $\xi$,  the tautological polarization on $X_e$,  and thus we get the exact sequence
\begin{eqnarray}\label{eq:coker-pullback}
\xymatrix{ 
0 \to \varphi^*\big(\sA (- c_1(\E_e))\big)\ar[r]^-{\overline{\phi}} \otimes \xi& \varphi^* \big(\sB(- c_1(\E_e))\big)\otimes \xi \to\varphi^*\big(\cH_r (-c_1)\big)\otimes \xi \to 0,}                                  
 \end{eqnarray} defining $\overline{\phi}$. Set  $\overline{\sA} := \varphi^*\big(\sA (- c_1(\E_e)\big)\otimes \xi$  and  $\overline{\sB} :=  \varphi^* \big(\sB(- c_1(\E_e)\big)\otimes \xi$.  Recall that the 
 modular component $\mathcal M(r)$ as in Theorem \ref{Main Theorem} has an open dense subset parametrizing isomorphism classes of slope-stable, rank-$r$ vector bundles $\mathcal U_r$, which are Ulrich w.r.t. the tautological polarization $\xi$ of $X_e$ 
 and with Chern classes determined by the iterative constructions as in \cite{fa-fl2} 
 (in particular, the first Chern class $c_1$ is as reminded in   Theorem \ref{Main Theorem}); for $[\mathcal U_r] \in \mathcal M(r)$ general it has also been proved in \cite[Proof of Thm.\,5.1]{fa-fl2} that the bundle $\mathcal U_r$ has in particular the splitting type requested by Proposition \ref{bijection}, namely $(\cU_r)_{|_F}\cong\cO_{\PP^1}(1)^{\oplus r}$, on any $\varphi$-fiber $F$. As a consequence of the bijective correspondence induced by Proposition \ref{bijection}, in 
 \cite{fa-fl2} we deduced therefore that 
 $\cU_r= \xi \otimes \varphi^*\big(\cH_r (-c_1(\mathcal E_e))\big)$, with $\mathcal H_r$ Ulrich w.r.t. $c_1(\mathcal E_e)$ on $\FF_e$ as above.

Then the sequence \eqref{eq:coker-pullback} reads
\begin{eqnarray}\label{eq:coker-pullback1}
\xymatrix{ 
0 \to\overline{\sA}\ar[r]^-{\overline{\phi}} & \overline{\sB}\to \cU_r\to 0.}                                  
 \end{eqnarray}  In particular, for those morphisms $\overline{\phi} \in \rm{Hom}_{X_e} ( \overline{\sA}, \overline{\sB})$ such that ${\rm coker}(\overline{\phi}) = \mathcal U_r$, one has that ${\rm coker}(\overline{\phi}) $ is locally free, of rank $r$ and it is moreover Ulrich on $X_e$ w.r.t. the tautological polarization $\xi$, with Chern classes 
 $c_i ({\rm coker}(\overline{\phi})) = c_i(\mathcal U_r)$, $1 \leq i \leq 3$, computed by iterative constructions of the vector bundles $\mathcal U_r$ as in \cite{fa-fl2} (e.g. $c_1$ is reminded in Theorem \ref{Main Theorem} above). \color{black}

Let ${\overline{\phi}_{gen}} \in  \rm{Hom}_{X_e} (\overline{\sA},\overline{\sB})$ be general;  since 
\begin{eqnarray*}
\overline{\sA}^{\vee} \otimes \overline{\sB} = \varphi^*\big(\sA^{\vee}\otimes \sB) = 
 \varphi^*\big(\mathcal O_{\FF_e}(0, 1)^{\oplus (\gamma\,\delta)} \oplus \mathcal O_{\FF_e}(1, e)^{\oplus (\gamma\,\tau)}\big),
 \end{eqnarray*} i.e. $\sA^{\vee}\otimes \sB$ is globally generated, so $\overline{\sA}^{\vee} \otimes \overline{\sB}$ is also globally  generated. Therefore, by \cite[Thm.\,4.2]{ba},  {(cf. also \cite[Thm.\,2]{ba2}) ${\overline{\phi}_{gen}}$ is injective and it gives rise to an exact sequence 
$$\xymatrix{ 
0 \to\overline{\sA}\ar[r]^-{{\overline{\phi}_{gen}}} & \overline{\sB}\to {\rm coker}({\overline{\phi}_{gen}})\to 0.}   
$$Since $\overline{\phi} \in \rm{Hom}_{X_e} ( \overline{\sA}, \overline{\sB})$ as in \eqref{eq:coker-pullback1} is such that ${\rm coker}(\overline{\phi})= \mathcal U_r$ is locally free, then also ${\rm coker}({\overline{\phi}_{gen}})$ is locally free, as locally freeness is an open condition on the (irreducible) vector space $ \rm{Hom}_{X_e} ( \overline{\sA}, \overline{\sB})$.   Moreover, the  rank of   ${\rm coker}({\overline{\phi}_{gen}})$ is  given by $\delta+\tau-\gamma=r = \rm{rank}(\cU_r)$,  with $\gamma,  \delta, \tau$ as in the proof of Theorem \ref{thm:moduli_UlrichFe}. Furthermore,  once again  from  the irreducibility of  $\rm{Hom}_{X_e} (\overline{\sA},\overline{\sB})$ and  from the constancy of Chern classes in irreducible flat families of vector bundles of given rank (or even from the fact that $\mathcal U_r$ and ${\rm coker}({\overline{\phi}_{gen}})$ are both locally free cokernels of injective vector bundle morphisms in $\rm{Hom}_{X_e} ( \overline{\sA}, \overline{\sB})$)  one  has that  
 \begin{eqnarray}\label{ci}
 c_i({\rm coker}({\overline{\phi}}_{gen}))= c_i(\cU_r)  \mbox{ for $0\le i\le3.$}
  \end{eqnarray}
Finally  since $\cU_r$ is Ulrich on $X_e$ w.r.t. $\xi$  we have
 \begin{eqnarray*}
h^i(\cU_r(-j\xi))= 0 \quad  \mbox{for $0\le  i \le 3$ and $1\le j \le 3$,} 
 \end{eqnarray*}
  then by semicontinuity 
   \begin{eqnarray*}
h^i({\rm coker}({\overline{\phi}_{gen}})(-j\xi))= 0\quad   \mbox{for $0\le  i \le 3$ and $1\le j \le 3$;}  
\end{eqnarray*}
hence ${\rm coker}({\overline{\phi}}_{gen})$ is Ulrich w.r.t. $\xi$.}

The fact that $\rm{Hom}_{X_e} (\overline{\sA},\overline{\sB})$  is irreducible implies that it must dominate the modular component $\cM(r)$ (as in Theorem \ref{Main Theorem}) containing $[\mathcal U_r]$ as its general point,  which  therefore implies that   $\cM(r)$ is unirational. The generic smoothness of   $\cM(r)$ as well as its dimension formula have already being proved in  Theorem \ref{Main Theorem} - (b), (c) (more precisely in \cite[Main Theorem]{fa-fl2}). 
 \end{proof}

%%%%%%%%%%%%%%
  %
  % BIBLIOGRAFIA
  %
  %%%%%%%%

\end{document}